\documentclass[12pt]{amsart}
\usepackage[latin9]{inputenc}
\usepackage{amstext}
\usepackage{amsthm}
\usepackage{amssymb}
\usepackage[unicode=true,
 bookmarks=false,
 breaklinks=false,pdfborder={0 0 1},backref=section,colorlinks=false]{hyperref}
\usepackage{verbatim}
\usepackage{amstext}
\usepackage{amsthm}
\usepackage{eurosym}
\usepackage{latexsym}
\usepackage{amsthm}
\usepackage{amsfonts}
\usepackage{ulem}
\usepackage{xcolor}

\theoremstyle{plain}
\newtheorem{theor}{Theorem}
\newtheorem{exttheorem}{Theorem}

\newtheorem{theorem}{Theorem}

\newtheorem{lemma}[theorem]{Lemma}
\newtheorem{cor}[theorem]{Corollary}
\newtheorem{proposition}[theorem]{Proposition}

\newtheorem{probl}[theorem]{Problem}
\theoremstyle{definition}

\newtheorem*{acknowledgements*}{Acknowledgements}
\theoremstyle{remark}
\newtheorem{remark}[theorem]{Remark}

\numberwithin{theorem}{section}   
\makeatother

\normalbaselines

\renewcommand{\leq}{\leqslant}
\renewcommand{\le}{\leqslant}
\renewcommand{\geq}{\geqslant}
\renewcommand{\ge}{\geqslant}

\newcommand\bbN{\mathbb{N}}
\newcommand\bbR{\mathbb{R}}
\newcommand\bbE{\mathbb{E}}
\newcommand\bbK{\mathbb{K}}

\newcommand\avg{\mathrm{avg}}
\newcommand\LFR{\mathrm{LFR}}
\newcommand\IR{(\mathrm{IR})}
\newcommand\eps{\varepsilon}

\DeclareMathOperator\sign{sign}
\DeclareMathOperator\supp{supp}
\DeclareMathOperator\conv{conv}

\begin{document}
\title{Rigidity of sets of independent functions in symmetric spaces}
\author{Sergey V. Astashkin}
\address{Astashkin: Department of Mathematics, Samara National Research University, Moskovskoye shosse 34, 443086, Samara, Russia; Lomonosov Moscow State University, Moscow, Russia; Department of Mathematics, Bahcesehir University, 34353, Istanbul, Turkey.}
\email{astash56@mail.ru}

\author{Yu. V. Malykhin}
\address{Malykhin: Steklov Mathematical Institute, 8 Gubkina st., Moscow,
119333, Russian Federation}
\email{malykhin-yuri@yandex.ru}

\thanks{The work of the first named author was performed at Lomonosov Moscow
State University and was supported by the Russian Science
Foundation, project no. 23-71-30001.}

\date{\today }

\maketitle

\begin{abstract}
We say that a symmetric function space $X$ has the $\IR$ property whenever all
    sets of $N$ independent mean zero functions $f_1,\ldots,f_N\in X$,
    $\|f_k\|_X\ge 1$, are poorly approximated by any linear combinations of arbitrary $n$
    functions, if $n$ is sufficienly smaller that $N$; namely, for some $\gamma=\gamma(X)>0$
    we have $d_n(\{f_1,\ldots,f_N\},X)\ge \gamma$, $n\le \gamma N$, where 
    $d_n(K,X)$ is the Kolmogorov $n$-width of the set $K\subset X$.
    The spaces $X=L_p$ satisfy this property if and only if
    $1\le p\le2$ or $p=\infty$.
    The goal of this paper is to move from $L_p$ scale to a larger class of
    symmetric spaces.  We obtain rather broad conditions, under which
    such a space $X$ has the $\IR$ property and prove precise statements for
    particular scales of Lorentz $L_{p,q}$ spaces and Orlicz spaces.
\end{abstract}

\section{Introduction}

Sequences of independent functions are studied both in probability theory and in
functional analysis. They play an important role in areas related to the
geometry of function spaces,  and in particular to the study of their lattice
and symmetric structures. This was demonstrated in full through the profound
results obtained by W.~B.~Johnson, B.~Maurey, G.~Schechtman, and L.~Tzafriri in
their memoir \cite{JMShT}. Additionally, let us mention some surveys, where
close problems are investigated. In~\cite{Gap66}, the function-theoretic
properties of sequences of both independent and weakly dependent (lacunary)
systems of functions are considered. In~\cite{AS}, the norms of sums 
of independent functions in symmetric spaces are investigated, while
in~\cite{A-24} the properties of subspaces generated by those are studied due to
the geometrical structure of these spaces. For the detailed information on the
simplest and alongside with that very important sequence of independent
functions, the Rademacher system, see the book~\cite{Abook}.

In this paper, we investigate the approximation properties of sets of
independent functions in function spaces, namely, the possibility of
approximating sets of $N$ such functions by elements from
linear subspaces of dimension essentially smaller than $N$. Let us give formal
definitions.

Let $X$ be a linear normed space. The Kolmogorov $n$-width
of a set $K\subset X$ is defined as the best approximation of $K$
by linear subspaces of $X$ of dimension at most $n$:
\begin{equation}
\label{base_0}
d_n(K,X) := \inf_{\dim V\le n} \sup_{x\in K}\rho(x, V)_X,
\quad\mbox{where }\rho(x,V)_X := \inf_{v\in V} \|x-v\|_X.
\end{equation}

Some lower bounds for the (Kolmogorov)
widths are based on the following simple fact: if elements $\varphi_1,\ldots,\varphi_N$~form an orthonormal system in an Euclidean space $E$, then for any $n$-dimensional subspace $V\subset E$ we have
\begin{equation}
\label{base_l2}
\rho(\varphi_1,V)_E^2 + \ldots + \rho(\varphi_N,V)_E^2 \ge N-n.
\end{equation}
This yields at once the lower bound in the well-known equality
$$
d_n(\{\varphi_1,\ldots,\varphi_N\},E) = (1-n/N)^{1/2}.$$
Thus, the orthonormal
system is poorly approximated by low-dimensional subspaces (say, if $n\le N/2$); we will informally refer to this property as to the rigidity of such a system.

Recently, the rigidity of sets in other spaces was studied in papers~\cite{M22,M24,MR25}. In particular, it was proved there that the classical Walsh system in the Paley numbering is not rigid in $L_p$ for $p<2$: the first $N$ functions of this system can be approximated with an error
$O(N^{-\delta_p})$ by a subspace of dimension $O(N^{1-\delta_p})$, where $\delta_p>0$.





To measure the rigidity of sets of independent functions in a space we introduce the following notion. Let $X$ be a symmetric space (an s.s.) on $[0,1]$ (see Section~\ref{sec_prelim} for all definitions). 
We will say that $X$ has the property (IR) ($X\in\IR)$ if
there exists $\gamma>0$ such that for any $N\in\mathbb N$ and arbitrary independent functions $f_1,\ldots,f_N\in X$ with $\int_0^1 f_k(x)\,dx=0$ and $\|f_k\|_X \ge 1$, $k=1,\ldots,N$, we have
$$
d_{\gamma N}(\{f_1,\ldots,f_N\},X) \ge \gamma.
$$
Informally, the fact $X\in\IR$ means that finite sets of mean zero independent functions are rigid in $X$.

It is easy to show that the negation of (IR) is equivalent to the existence of a sequence
$(f_k)_{k=1}^\infty\subset X$ of mean zero independent functions, $\|f_k\|_X\ge 1$,
which is well approximated by low-dimensional subspaces, i.e., for any
$\eps>0$ it holds:
$$
\lim_{N\to\infty} d_{\eps N}(\{f_1,\ldots,f_N\}, X) = 0.
$$

In~\cite{M24}, the rigidity of sets of mean zero independent functions in the space
$L_1$ and its absence in $L_p$ for $2<p<\infty$ were established; see~\cite[Corollary 3.1 and Proposition 4.2]{M24}. Moreover, as was proved in~\cite[\S1.2]{MR25}, such sets are rigid in $L_p$ for $1<p\le 2$. For the special case $p=\infty$, see~Proposition~\ref{prop_linf} below.
As a result, for $L_p$-spaces we have the following criterion.

\begin{exttheorem}
\label{extth}
The space $L_p[0,1]$, $1\le p\le\infty$, has the property (IR) if and
only if $1\le p\le 2$ or $p=\infty$.
\end{exttheorem}
Below, in Corollary~\ref{cor_lq}, we will give more precise estimates related to this special case.

The main goal of this paper is to move from $L_p$-spaces to a larger class of
symmetric function spaces by recovering the relationship between the
possibility of approximating subsets of independent functions in an s.s. $X$ by
low-dimension subspaces and the geometrical structure of $X$. It is worth
mentioning that in the setting of symmetric {\it sequence} spaces similar
estimates for various types of widths have been studied recently rather
intensively; see, for instance, \cite{DMM1,DMM2,DMM3,HM}). Note also that there
is a close relationship between the study of the  widths of finite systems of
functions and those of sets in $\bbR^N$ (see~\cite{M24,MR25}).

As we will show below, much in this more general situation is determined by lattice properties of the given space and especially by the structure of subspaces generated by disjoint functions. Our approach to the problem is based primarily on
the use of estimates related to the comparison of norms of sums of independent functions and their disjoint copies in s.s.'s (see \cite{JSh1,Brav,AS,ASZ-14}). We use also classical estimates for the Kolmogorov width of the Euclidean ball (see~\S\ref{Prel_width}) and some  methods developed recently in the papers~\cite{M24,MR25,Mdiss}. As a result, we obtain rather broad conditions, under which an s.s. $X$ has (respectively, does not have) the (IR) property. Based on these, we prove more precise statements for particular scales of s.s.'s, namely, for Lorentz $L_{p,q}$-spaces and Orlicz spaces.

Let us formulate several main results here (more precise estimates for the corresponding widths are contained below in Section~\ref{sec_main}).

\begin{theor}
    \label{th_intro_rigid}
    Let an s.s. $X$ on $[0,1]$ satisfy a lower 2-estimate, $X\supset L_2$, 
and let the associated space $X'$ have the Kruglov property. Then $X\in\IR$.
   \end{theor}

As a consequence, we obtain that the Lorentz spaces $L_{p,q}$ have  property (IR) if $1<p<2$, $1\le q\le 2$.

Observe that in all the cases, when we are able to prove that an s.s. $X$ has
the property (IR), we get additionally estimates from below for the widths of
dimension $n=N(1-\eps)$ for any fixed $\eps \in (0,1)$.

Now we present results of the opposite nature, insuring that an s.s. contains arbitrarily long sequences of mean zero independent 
functions that can be well approximated by low-dimensional subspaces.

\begin{theor}
    \label{th_intro_phi}
   Let $X$ be an s.s., $X\ne L_\infty$, and let the fundamental function $\phi_X$ of $X$ satisfy the condition:
$$
\varlimsup_{t\to 0+}{\phi_X(t)}t^{-1/2} = \infty.
$$
Then, $X\not\in\IR$.
\end{theor}

The main assumption of the second result of such a sort  is directly related to the structure of the set of all $q$, for which the given space $X$ contains uniformly embedded $l_q$-subspaces of arbitrarily high dimension (see the more precise Theorem~\ref{th_LFR} below) .
\begin{theor}
\label{th_LFR_a}
Assume that an s.s. $X$ has the Kruglov property and for some $q\in(2,\infty]$
    the space $\ell_q$ is roughly lattice finitely representable in $X$. Then
    $X\not\in\IR$.
\end{theor}
As a consequence, we obtain that an s.s. $X$ does not have the (IR) property whenever its lower Boyd index $\alpha_X$ belongs to the interval $(0,1/2)$. Moreover, Theorem~\ref{th_LFR_a} implies that the Lorentz spaces $L_{p,q}$ do not belong to the class (IR) if $\max\{p,q\}>2$.
The question whether the spaces $L_{2,q}$, $1\le q<2$, possess the
(IR) property, remains open.

In Section~\ref{sec_prelim}, we recall the necessary definitions and auxiliary results.
Section~\ref{sec_main} contains proofs of the theorems formulated in the Introduction and also some others, as well as open questions.

\section{Preliminaries}
\label{sec_prelim}


\subsection{Independent functions}
\label{Prel1}

The main object of our study are families of independent functions. Recall that
measurable functions (random variables) $f_1,f_2,\dots,f_N$, defined on
a probability space $(\Omega,\Sigma,{\mathbb P}),$ are called
{\it independent} if for any intervals $I_k$, $k=1,2,\dots,N$, from the real line the following equality holds:
$$
{\mathbb P}\{\omega\in \Omega:\,f_k(\omega)\in
I_k,\,k=1,2,\dots,N\}\;=\;\prod_{k=1}^N \mathbb P\{\omega\in
\Omega:\,f_k(\omega)\in I_k\}.$$ 

Measurable functions $f$ and $g$ are said to be {\it identically distributed} whenever we have
$$
{\mathbb P}\{\omega\in \Omega:\,f(\omega)>\tau\}={\mathbb P}\{\omega\in \Omega:\,g(\omega)>\tau\}\;\;\mbox{for each}\;\;\tau\in{\mathbb{R}}.$$
A function $f$ is said to be {\it symmetrically distributed} on $\Omega$ if $f$ and $-f$ are identically distributed.
Clearly,   such a function $f$ is mean zero on $\Omega$, i.e., 
$$\bbE f := \int_\Omega f(\omega)\,d{\mathbb P}(\omega)=0.$$

In what follows, we consider, mainly, measurable functions on the interval $[0,1]$ with Lebesgue measure $m$, defined on the $\sigma$-algebra of Lebesgue measurable sets. 

One of the most important examples of sequences of identically and symmetrically distributed independent functions on $[0,1]$ is the sequence of Rademacher functions $r_k(t)=\sign\sin 2^k\pi t$, $k=1,2,\dots$.

For more detailes on properties of systems of independent functions in function spaces, see, for instance, \cite[Chapter~2]{KS}.

\subsection{Symmetric spaces}
\label{Prel_symm}

Real-valued (Lebesgue) measurable functions $x(t)$ and $y(t)$ on $[0,1]$ are said to be {\it equimeasurable} if the functions $|x(t)|$ and $|y(t)|$ are  identically distributed (with respect to the Lebesgue measure).
In particular, every measurable function $x(t)$ is equimeasurable with its non-increasing left-continuous {\it rearrangement} $x^{\ast}(t)$ defined by
$$
x^{\ast}(t):=\inf \{\tau \geq 0\;:\; m\{s\in [0,1]\colon |x(s)| > \tau \} < t\},
\quad 0<t\le 1.
$$

A Banach function space $X$ of measurable functions (equivalence classes of) on $[0,1]$ is called {\it
symmetric} (s.s.) if 1) from $x \in X$ and $|y(t)|\le |x(t)|$ a.e., where $y(t)$ is measurable, it follows that $y\in X$ and ${\|y\|}_X \le {\|x\|}_X;$ 2) if $x \in X$ and the functions $x$ and $y$ are equimeasurable, it follows that $y\in
X$ and ${\|y\|}_X = {\|x\|}_X$.

Without loss of generality, we will assume that $\|\chi_{[0,1]}\|_X=1$ (in what follow, $\chi_A$ is the characteristic function of a set $A$). Then, for every s.s. $X$ on $[0,1]$ the one-norm embeddings $L_\infty[0,1]\stackrel{1}{\subset} X\stackrel{1}{\subset} L_1[0,1]$ ~hold \cite[Theorem II.4.1]{KPS}.

If $X$ is an s.s., then the {\it associated (or K\"{o}the dual)} space $X'$ consists of all measurable functions $y$ for which
$$
\|y\|_{X'}:=\sup\,\Bigl\{\int_0^1{x(t)y(t)\,dt}:\;
\|x\|_{X}\,\leq{1}\Bigr\}<\infty.
$$
The space $X'$ is also symmetric; it is isometrically embedded into the (Banach) dual $X^*$, and $X'=X^*$ if and only if $X$ is separable (see, e.g.,~\cite[\S\,II.4.5]{KPS}). Moreover, the space $X'$ is {\it maximal}
(or has the {\it Fatou property}), i.e., from the conditions that $x_n\in X'$, $n=1,2,\dots,$ $\sup_{n=1,2,\dots}\|x_n\|_{X'}<\infty$ and
$x_n\to{x}$ a.e. it follows that $x\in X'$ and $||x||_{X'}\le \liminf_{n\to\infty}{||x_n||_{X'}}.$

Let $X$ be an s.s. on $[0,1]$. By $\phi_X$ we will denote the {\it fundamental function} of $X$, which is defined by $\phi_X(t):=\|\chi_{(0,t]}\|_X$.
Also, the contraction operator ${\sigma}_\tau x(t):=x(t/\tau)\chi_{(0,\tau)}(t)$, $0\le t\le 1$ is bounded in $X$ for any $\tau\in(0,1]$. The {\it lower Boyd index} of $X$ is defined as follows:
$$
\alpha_X:=\lim\limits_{\tau \to 0+}\frac{\ln {\|{\sigma}_\tau\|}_{X \to X}}{\ln \tau}.
$$
For every s.s. $X$ we have $0\leq\alpha_X\le 1$.

Recall now the definition of some classes of s.s.'s, which will be duscussed in more detail in Section  \ref{sec_main}.

Let $1<p<\infty$, $1\le q\le\infty$. The Lorentz space $L_{p,q}$ consists of all measurable functions $x(t)$ on $[0,1]$ such that
$$
\|x\|_{p,q}:=\Big(\int_0^1 x^*(t)^q\,d(t^{q/p})\Big)^{1/q}<\infty$$
(with a natural modification for $q=\infty$). Though the functional $x\mapsto \|x\|_{p,q}$ for $q>p$ is not subadditive, the space $L_{p,q}$ may be equipped  with an equivalent symmetric norm for all $p$ and $q$.
Moreover, $L_{p,q_1}\subset L_{p,q_2}$ if $1\le q_1\le q_2\le\infty$, and $L_{p,p}=L_p$, $1<p<\infty$, isometrically \cite[Lemma~II.6.5]{KPS}.

Another natural generalization of $L_p$-spaces is the family of Orlicz spaces.
Let $\Phi$ be an {\it Orlicz function}, i.e., a (strictly) increasing, convex, continuous function on the half-axis $[0, \infty)$ such that $\Phi(0) = 0$, $\Phi(1)=1$.
The {\it Orlicz space} $L_\Phi$ consists of all measurable functions $x(t)$ on $[0,1]$, for which the following Luxemburg norm is finite: 
$$
\| x \|_{L_\Phi}: = \inf \left\{\lambda > 0 \colon \int_0^1 \Phi\Big(\frac{|x(t)|}{\lambda}\Big) \, dt \leq 1 \right\}.
$$
In particular, if $\Phi(u)=u^p$, $1\le p<\infty$, we have $L_\Phi=L_p$ with the usual norm.

Note that the definition of the space $L_\Phi$ depends (up to
norm equivalence) only on the behavior of the function $\Phi$ at infinity
(i.e., for large values of the argument).

If $\Phi$ is an Orlicz function, then its {\it Yang conjugate function} $\Phi'$ is defined as follows:
$$
\Phi'(u):=\sup_{t>0}(ut-\Phi(t)),\;\;u>0.
$$

Further, we will need the following expression for the lower Boyd index of the space $L_\Phi$:
\begin{equation}
\label{index}
    \alpha_{L_\Phi}={\sup}\Big\{\frac1{p}:\,\sup_{a,u\ge 1}\frac{\Phi(au)}{\Phi(u)a^p}<\infty\Big\}
\end{equation}
(see, e.g., \cite[Proposition 2.b.5]{LT}).

Further information on s.s.'s see in the books \cite{KPS,LT}.

\subsection{Lattice properties of symmetric spaces.}
\label{Prel3}

As it follows from the definition, every s.s. is a Banach function lattice with respect to
the usual a.e. order. All the concepts of this subsection pertain also to this more general class of spaces.

Let $1\le p\le\infty$. We say that an s.s. $X$ satisfies a lower (resp. an
upper) $p$-estimate if there exists a constant $C_X>0$ such that for all $N\in\bbN$ and any pairwise disjoint functions $x_k\in X$, $k=1,\dots,N$, the following holds:
$$
\Big\|\sum_{k=1}^N x_k\Big\|_X\ge C_X^{-1}\Big(\sum_{k=1}^N \|x_k\|_X^p\Big)^{1/p}$$
(resp.
$$
\Big\|\sum_{k=1}^N x_k\Big\|_X\le C_X\Big(\sum_{k=1}^N \|x_k\|_X^p\Big)^{1/p})
$$
(with natural modification for $p=\infty$).
Similarly, an s.s. $X$ is $p$-concave (respectively, $p$-convex) if
$$
 \Big\|(\sum_{k=1}^N |x_k|^p)^{1/p}\Big\|_X\ge D_X^{-1} \Big(\sum_{k=1}^N \|x_k\|^p\Big)^{1/p}$$
(resp.
$$
 \Big\|(\sum_{k=1}^N |x_k|^p)^{1/p}\Big\|_{X}\le D_X \Big(\sum_{k=1}^N \|x_k\|^p\Big)^{1/p}\;),$$
for any $x_1,\ldots,x_N\in X$.
Clearly, each $p$-concave (resp. $p$-convex) s.s. $X$ satisfies a lower (resp. an upper) $p$-estimate.

Let $1\le q\le \infty$. Then, $\ell_q$ if $1\le q<\infty$ ($c_0$ if $q=\infty$) is said to be  {\it roughly lattice finitely representable} in an s.s. $X$ if there exists a constant $K>0$ such that for every $N\in\mathbb{N}$ one can find pairwise disjoint functions $x_k\in X$, $k=1,2,\dots,N$, which satisfy, for arbitrary $a_k\in\mathbb{R}$,  the following condition:
\begin{equation}
\label{first}
K^{-1}\|(a_k)\|_{\ell_q^N}
\le \Big\|\sum_{k=1}^N a_kx_k\Big\|_X
\le K\|(a_k)\|_{\ell_q^N}.
\end{equation}

A given s.s. $X$ denote by $\LFR(X)$ the set of all $q\in [1,\infty]$ such that $\ell_q$ is roughly
lattice finitely representable in $X$. According to results of the paper \cite{Shepp},
$\sup \LFR(X)$ and $\inf \LFR(X)$ belong to the set $\LFR(X)$
for every Banach lattice. Moreover, if $X$ is an s.s., then $1/\alpha_X\in\LFR(X)$~\cite[Theorem~2.b.6]{LT} (for a full proof of this fact, see \cite[Theorem 4]{A-11}).

It is clear that the space $L_q$ satisfies a lower (resp. an upper) $p$-estimate if and only if $q\le p\le\infty$ (resp. $1\le p\le q$), and also that $\LFR(L_p)=\{p\}$.

\vskip0.1cm

\subsection{Kruglov property and comparison of norms of sums of independent functions and their disjoint copies.}
\label{Prel4}

Further, the following property of s.s.'s will play an important role.

Let $f$ be a measurable function (random variable) on $[0,1].$ By $\pi(f)$ we denote the sum $\sum_{k=1}^N f_k$, where $f_k$~are independent copies of $f$, and let $N$~be a random variable, having a Poisson distribution with parameter 1 and independent of
the sequence $\{f_k\}$.
A s.s. $X$ on $[0,1]$ is said to have the {\it Kruglov property} $(X\in \bbK)$ if from $f\in X$ it follows that $\pi(f)\in X$.

Somewhat simplifying, we can say that an s.s. $X\in \bbK$ if it lies sufficiently ``far''\:from the space $L_\infty$.
In particular, every maximal s.s. $X$, which contains the space $L_p$ for some $p<\infty$, has the Kruglov property  \cite[Theorem~I.2]{Brav} (and hence $X\in \bbK$ if $\alpha_X>0$ \cite[Theorem~I.3]{Brav}). Moreover, if $\mathrm{Exp}L_q$ is the Orlicz space generated by an Orlicz function equivalent to the function $e^{u^q}$ for large $u>0$, then $\mathrm{Exp}L_q\in \bbK$ if and only if $0<q\le 1$ (for more details on what symmetric spaces have the Kruglov property, see \S\,4.3 in the survey \cite{AS}).

From the results of \cite{JSh1} (see also \cite{AS} and, in the more general quasi-normed case, \cite{ASZ-14}) if follows that the condition  $X\in \bbK$ ensures that the $X$-norms of sums of independent functions may be estimated from above (up to a constant) by those of their disjoint copies. To give a precise formulation of this statement, we recall the following important definition (see \cite{JMShT} or \cite[2.f]{LT}). If $X$ is an s.s. on $[0,1]$, then the set
$Z_X^2$ consists of all measurable functions $f$ on $[0,\infty)$ such that
\begin{equation}
\label{main ineq1dop}
\|f\|_{Z_X^2}:=\|f^*\chi_{[0,1]}\|_X+\|f^*\chi_{[1,\infty)}\|_{_{L_2[1,\infty)}}<\infty.
\end{equation}
The set $ Z_X^2$ equipped with a suitable norm becomes a symmetric space on the half-axis $[0,\infty)$ (the definition of s.s.'s for this case is similar to that for the interval $[0,1]$, see \cite[Chapter II]{KPS}). 

By \cite{JSh1} (see also \cite[\S\,6]{AS}), we have the following: if $X\in \bbK$, $f_k\in X$ are independent functions such that $\int_0^1 f_k(t)\,dt=0$, $k=1,\dots,N$, then 
\begin{equation}
    \label{main ineq1}
\Big\|\sum_{k=1}^N f_k\Big\|_{X}\le C(X) \Big\|\sum_{k=1}^N \bar{f}_k\Big\|_{Z_{X}^2},
\end{equation}
where $\bar{f}_k(t)=f_k(t-k+1)\chi_{[k-1,k]}(t),$ $k=1,2,\dots,N$, 
and the constant $C(X)$ depends only on $X$ (note that the opposite inequality holds for any s.s.). In what follows, we will make
use of the following consequence of inequality  \eqref{main ineq1}: if $X\in \bbK$, then the estimate
\begin{equation}
\label{main ineq2}
\Big\|\sum_{k=1}^N f_k\Big\|_{X}\le B \Big(\sum_{k=1}^N \|f_k\|_X^2\Big)^{1/2}
\end{equation}
is satisfied for some $B>0$ and all mean zero independent $f_k\in X$, $k=1,\dots,N$, if and only if $X$ satisfies an upper $2$-estimate and $X\subset L_2$ (see \cite[Theorem~II.2.4]{Brav} or \cite[Corollary~40]{ASZ-14}).

\vskip0.1cm

\subsection{Widths.}
\label{Prel_width}

All the results on widths mentioned below can be found, for example, in \cite[Chapter 13]{LGM} and \cite{MR25}.

Let $1\le q\le\infty$ and let $B_q^N$ denote the unit ball of the space $\ell_q^N$. Thus, $B_1^N$ is the octahedron $\conv\{\pm e_k\}_{k=1}^N$ (here, $e_1,\ldots,e_N$ are the standard basis vectors of $\bbR^N$) and $B_\infty^N$ is the cube $[-1,1]^N$.

Recall that the definition of the Kolmogorov width (see \eqref{base_0}) involves linear subspaces, which contain, in particular, the null vector. Therefore, the Kolmogorov width of a set $K$ does not change when replacing any elements of $K$ with the opposite ones.
In what follows, we will use the following simple relations:
$$
d_n(K, X) = d_n(\conv K, X)
$$
and
\begin{equation}
    \label{width_K_K}
d_n(K-K,X) \le 2d_n(K,X),
\quad\mbox{where $K-K:=\{x-y\colon x,y\in K\}$.}
\end{equation}
Observe that the first equality implies that
$d_n(\{e_1,\ldots,e_N\},\ell_q^N)=d_n(B_1^N,\ell_q^N)$.

Next, we will also need estimates for the widths of the Euclidean ball. It is known that
\begin{equation}
\label{width_b2_precise}
d_n(B_2^N,\ell_\infty^N) \le Cn^{-1/2}\ln^{1/2}(2N/n),
\end{equation}
which implies the following consequence: for any $\gamma\in (0,1)$
\begin{equation}
    \label{width_b2}
    d_n(B_2^N,\ell_\infty^N) \le C(\gamma)N^{-1/2}\quad\mbox{if $n\ge \gamma
    N$.}
\end{equation}
From well-known estimates of the octahedron's widths (or, otherwise,  from~\eqref{width_b2_precise}) it follows that for $q>2$
\begin{equation}
\label{width_b1a}
d_n(B_1^N,\ell_q^N) \le C(q)N^{-\delta}
\quad\mbox{if $n\ge N^{1-\delta}$,}
\end{equation}
where $\delta=\delta(q)>0$ is sufficiently small.
We will also use the simple equality
\begin{equation}
    \label{width_cube}
    d_n(B_\infty^N,\ell_\infty^N) = 1,
    \quad n<N.
\end{equation}

 Observe that the Kolmogorov width $d_n(K,X)$ is estimated from below by the "averaged width"\:$d_n^\avg(K,X)$, in which the supremum over elements $x\in K$ is replaced by the average with respect to some measure. In particular, for $1\le q<\infty$ and for a finite set $\{x_1,\ldots,x_N\}\subset X$ we put
    $$
    d_n^\avg(\{x_1,\ldots,x_N\},X)_q
    := \inf_{\dim V\le n}\left(\frac1N\sum_{k=1}^N \rho(x_k,V)^q\right)^{1/q}.
    $$
    
       Let $X,Y$ be Banach spaces and let $T:X\to Y$~be a bounded linear operator, $K\subset X$. Then,
    \begin{equation}
        \label{width_T}
        d_n(T(K),Y) \le \|T\|\cdot d_n(K,X),
    \end{equation}
   and also for arbitrary $x_1,\ldots,x_N\in X$
    \begin{equation}
        \label{avg_width_T}
        d_n^\avg(\{Tx_1,\ldots,Tx_N\},Y)_q \le \|T\|\cdot
        d_n(\{x_1,\ldots,x_N\},X)_q.
    \end{equation}

  In the proof of Theorem   \ref{th_intro_rigid}, we will need the following simple lemma, a special case of Theorem 1.1 from~\cite{Mdiss}. 
    \begin{lemma}
        \label{lem_width}
        Assume that $\{x_k\}_{k=1}^N$ is a linearly independent system of elements in a linear normed space $X$, the system $\{x_k^*\}_{k=1}^N\subset X^*$ is conjugate to $\{x_k\}_{k=1}^N$, and for some $B>0$ the following inequality holds:
         \begin{equation}
            \label{dual_cond}
        \Big\|\sum_{k=1}^N a_kx_k^*\Big\|_{X^*}
        \le B \Big(\sum_{k=1}^N a_k^2\Big)^{1/2},
        \quad\forall a_1,\ldots,a_N \in\mathbb R. 
        \end{equation}
        Then, we have
        $$
            d_n^\avg(\{x_1,\ldots,x_N\}, X)_2
            \ge B^{-1}(1-n/N)^{1/2},
            \quad 1\le n\le N.
        $$
    \end{lemma}

For completeness, we present a proof of this assertion here.
\begin{proof}
Consider the "synthesis"\:and "analysis"\:operators $S\colon\ell_2^N\to X^*$ and $A\colon X\to\ell_2^N$ defined by
        $$
        S\colon (a_1,\ldots,a_N)\mapsto \sum_{k=1}^N a_kx_k^*,
        \quad A\colon x\mapsto (\langle x,x_1^*\rangle,\ldots,\langle x,x_N^*\rangle).
        $$
Assumption~\eqref{dual_cond} means that $\|S\|\le B$.
Moreover, it is easy to verify that $A^*=S$, which implies $\|A\|=\|S\|\le B$.
Therefore, applying~\eqref{avg_width_T} to the operator $A$ and the set $\{x_1,\ldots,x_N\}$, we obtain:       
        $$
        d_n^\avg(\{e_1,\ldots,e_N\},\ell_2^N)_2
        \le \|A\|\cdot d_n^\avg(\{x_1,\ldots,x_N\},X)_2,
        $$
where $\{e_k\}_{k=1}^N$ is the standard basis in $\mathbb R^N$.
Since, by virtue of~\eqref{base_l2}, the width on the left-hand side
of this inequality is not less than $(1-n/N)^{1/2}$, the required estimate follows.
\end{proof}

 Finally, for a non-integer $r$, we set $d_r(K,X):=d_{\lfloor r\rfloor}(K,X)$ and similarly for the averaged width.

\vskip0.3cm

    \section{Main results}
    \label{sec_main}

\subsection{Spaces with the (IR) property}

    \begin{proof}[Proof of Theorem~\ref{th_intro_rigid}]

    First of all, note that the space $X$ is separable. Indeed, otherwise $X$ contains a subspace isomorphic to
    $\ell_\infty$, and the isomorphism can be chosen so that the unit vectors from $\ell_\infty$ correspond to some
    pairwise disjoint functions in $X$ (see, e.g., \cite[Proposition 1.a.7]{LT}). This, however, contradicts
    the assumption that $X$ satisfies a lower $2$-estimate.

    Since $X$ is separable, the dual space $X^*$ coincides with the associated space $X'$, and hence is also an
    s.s. on $[0,1]$.
    Therefore, by duality, for each $k=1,\dots,N$ there exists a function $g_k'\in X'$ such that $\int_0^1
        f_kg_k'\,dt=1$ and $\|g_k'\|_{X'}=1$.  Set
    $$
    g_k:=\mathbb{E}_{\mathcal{A}_k}g_k'-\mathbb{E}g_k',\;\;k=1,\dots,N,$$
    where $\mathbb{E}_{\mathcal{A}_k}$ is the conditional expectation operator with respect to the $\sigma$-algebra of
    subsets of $[0,1]$ generated by the function $f_k$. Then, first,
    $$
    \mathbb{E}g_k=\mathbb{E}\mathbb{E}_{\mathcal{A}_k}g_k'-\mathbb{E}g_k'=
    \mathbb{E}g_k'-\mathbb{E}g_k'=0.$$
    Second, as is well known, every conditional expectation operator has norm one, both in the spaces $L_1$ and
    $L_\infty$, and hence in every maximal s.s. \cite[Theorem II.4.9]{KPS}. Consequently, since the space $X'$ is maximal (see
    Section~\ref{Prel_symm}), we have
    $$
    \|g_k\|_{X'}\le \|\mathbb{E}_{\mathcal{A}_k}g_k'\|_{X'}+\|\mathbb{E}g_k'\|_{X'}\le 2,\;\;k=1,\dots,N.$$
    Moreover, it follows from the definition that the functions $g_k$, $k=1,\dots,N$, are independent, 
    $$
    \int_0^1 f_kg_k\,dt=\int_0^1 f_k\mathbb{E}_{\mathcal{A}_k}g_k'\,dt-\int_0^1 f_k\,dt\cdot\mathbb{E}g_k'
    =\int_0^1 \mathbb{E}_{\mathcal{A}_k}(f_kg_k')\,dt=\int_0^1f_kg_k'\,dt=1$$
    and, if $j\ne k$, then by the independence of $f_j$ and $\mathbb{E}_{\mathcal{A}_k}g_k'$, we have
    $$
    \int_0^1 f_jg_k\,dt=\int_0^1 f_j\,dt\cdot \int_0^1 \mathbb{E}_{\mathcal{A}_k}g_k'\,dt-
    \int_0^1 f_j\,dt\cdot\mathbb{E}g_k'=0.$$
    Thus, $\{g_k\}_{k=1}^N$ is a system conjugate to $\{f_k\}_{k=1}^N$.

    Next, since $X$ satisfies a lower $2$-estimate and $X\supset L_2$, it follows that $X'$ satisfies an upper
    $2$-estimate \cite[Proposition 1.f.5]{LT} and $X'\subset L_2$. Taking into account that $X'\in \bbK$, by
    \cite[Theorem~II.2.4]{Brav} (see also \cite[Corollary~40]{ASZ-14}), we obtain the following $\ell_2$-estimate of
    the form \eqref{main ineq2} for some $B>0$ and all $a_k\in\mathbb{R}$:
    \begin{equation*}
    \Big\|\sum_{k=1}^N a_kg_k\Big\|_{X'}
    \le B \Big(\sum_{k=1}^N a_k^2\|g_k\|_{X'}^2\Big)^{1/2}
    \le 2B \Big(\sum_{k=1}^N a_k^2\Big)^{1/2}.
        \end{equation*}
    Combining this with Lemma~\ref{lem_width}, we infer
     \begin{equation*}
            \label{rigid_ineq}
        d_n^{\avg}(\{f_1,\dots,f_N\},X)_2\ge c (1-n/N)^{1/2}.
        \end{equation*}
    with the constant $c=1/(2B)$.

    It remains to observe that the fact that $X$ has the property (IR) follows directly from the last inequality.
    \end{proof}


    \begin{theor}
    \label{main1}
    Let $X$ be an s.s. on $[0,1]$. Suppose that $X$ is separable and there exists $B=B(X)>0$ such that for any
    finite collection of independent functions $\{g_k\}_{k=1}^N\subset X'$ with $\int_0^1 g_k(t)\,dt=0$,
    $k=1,\dots,N$, the following estimate holds:
        \begin{equation}
     \label{Bahr-Ess}
    \Big\|\sum_{k=1}^N g_k\Big\|_{X'}
    \le B \Big(\sum_{k=1}^N \|g_k\|_{X'}^2\Big)^{1/2}.
        \end{equation}

    Then for any finite collection of mean zero independent functions $\{f_k\}_{k=1}^N\subset X$ with 
    $\|f_k\|_X\ge 1$, $k=1,\dots,N$, and any $n<N$, we have
         \begin{equation*}
         \label{Bahr-Esseen_a}
           d_n^{\avg}(\{f_1,\dots,f_N\},X)_2\ge (2B)^{-1} (1-n/N)^{1/2}.
    \end{equation*}

        In particular, the statement is true if at least one of the following conditions is satisfied:
        
        (i) $X$ satisfies a lower $2$-estimate, $X\supset L_2$, and $X'\in \bbK$;
        
        (ii) $X$ is $2$-concave and $X'\in \bbK$.
    \end{theor}
 \begin{proof}
 As a straightforward analysis of the proof of Theorem~\ref{th_intro_rigid} shows, it suffices to verify that condition (ii) implies inequality \eqref{Bahr-Ess}.
  
Applying the version of Khintchine's inequality for s.s.'s with the Kruglov property proved in \cite[Theorem 1]{A-08}
(see also \cite[Theorem 23]{ASZ-14} for quasi-Banach spaces) to the functions $g_k$ satisfying the conditions of the
theorem, we obtain:
    $$
 \Big\|\sum_{k=1}^N g_k\Big\|_{X'}\le C(X) \Big\|\Big(\sum_{k=1}^N g_k^2\Big)^{1/2}\Big\|_{X'}.
    $$
 Since the space $X'$ is $2$-convex \cite[Proposition 1.d.4]{LT}, inequality \eqref{Bahr-Ess} follows.
\end{proof}

    \begin{remark}
    An estimate analogous to \eqref{Bahr-Ess} was first obtained in 1965 in the paper \cite{BE} (see Theorem~2
    therein). More precisely, B.~von Bahr and C.~Esseen proved the following: if $1\le p\le 2$, $N\in\mathbb N$, and
    $\{f_k\}_{k=1}^N\subset L_p[0,1]$ is a sequence of independent functions with $\int_0^1 f_k(t)\,dt=0$,
    $k=1,\dots,N$, then
        \begin{equation*}\label{Bahr-Esseen}
            \Big\|\sum_{k=1}^N f_k\Big\|_{p}\le\Big(2\sum_{k=1}^N\|f_k\|_{p}^p\Big)^{1/p}.
    \end{equation*}
    Later, in \cite[\S\,II.2]{Brav}, a connection between such estimates and the Kruglov property was revealed, thanks to which similar results were obtained for (Banach) s.s.'s (an extension to the quasi-Banach setting see in \cite{ASZ-14}).

    \end{remark}
    
\begin{remark}
   Suppose $X$ is an s.s. such that estimate \eqref{Bahr-Ess} holds in the
    associated space $X'$. Then $X$
   satisfies a lower $2$-estimate and $X\supset L_2$ \cite[Theorem 38(b)]{ASZ-14}.
    \end{remark}

    \begin{remark}
      Let $X$ be a separable s.s. Inequality~\eqref{Bahr-Ess} holds trivially with $B=N^{1/2}$, and therefore, by Theorem \ref{main1}, we have         
        $$
        d_{N-1}(\{f_1,\ldots,f_N\},X)\ge (2N)^{-1}$$
        for every mean zero independent functions $f_1,\ldots,f_N$ from $X$. Hence, when verifying
       the property (IR) or similar inequalities, we may assume that $N$ is sufficiently large.
    \end{remark}

     \begin{remark}
         It is of interest to compare condition (ii) of Theorem~\ref{main1} with the result of
         \cite[Corollary 5.3, (ii)]{HM}, which implies that the basis vectors are ``rigid'' in any symmetric
         $2$-concave sequence space $X$:
         $$
         d_n(\{e_1,\ldots,e_N\},X) \ge c(X,\eps)\quad\text{for } n\ge N(1-\eps)\text{ and all } N.
         $$
     \end{remark}
    
    We now turn to concrete classes of s.s.'s. First, we refine the (IR) property for $L_p$-spaces.
     
    \begin{proposition}
        \label{prop_linf}
        Let $f_1,\ldots,f_N\in L_\infty[0,1]$ be independent, $\int_0^1 f_k(x)\,dx=0$, and $\|f_k\|_\infty \ge 1$,
        $k=1,\ldots,N$. Then
        $$
        d_{N-1}(\{f_1,\ldots,f_N\},L_\infty) \ge 1/2.
        $$
    \end{proposition}

    \begin{proof}
        Suppose that a good approximation of the family $\{f_k\}_{k=1}^N$ by an $n$-dimensional subspace is possible,
        i.e., the inequality
        \begin{equation}
            \label{linf_proof_approx}
        \left\|f_k - \sum_{j=1}^n a_{k,j}g_j\right\|_\infty \le 1/2-\eps,
        \quad k=1,\ldots,N,
        \end{equation}
        holds for some functions $\{g_j\}_{j=1}^n$, where $n<N$, a coefficient matrix $A=(a_{k,j})$, and $\eps>0$. We
        will arrive at a contradiction.
       
        Denote
        $$
        \vec{f}(t):=(f_1(t),\ldots,f_N(t)),
        \quad \vec{h}(t):=\left(\sum_{j=1}^n a_{1,j}g_j(t),\ldots,\sum_{j=1}^n
        a_{N,j}g_j(t)\right)=A\vec{g}(t).
        $$

        Fix a small $\delta>0$. By the condition $\|f_k\|_\infty \ge 1$, we may assume that
        $$
        m\{f_k(t)\ge 1-\delta\}>0\;\;\mbox{for all}\;k$$ 
        (if necessary, we replace $f_k$ by $-f_k$; this does not change the
        width).
        Since the $f_k$ are mean zero, we also have $m\{f_k(t)\le 0\}>0$.
       Putting
       $$M_k^0:=\{t\colon f_k(t)\le 0\}\;\;\mbox{and}\;\;M_k^1:=\{t\colon f_k(t)\ge 1-\delta\},$$
        for each tuple $\tau=(\tau_1,\ldots,\tau_N)\in\{0,1\}^N$ we form the set
        $M^\tau := \bigcap_{k=1}^N M_k^{\tau_k}$. Observe that, by the independence of the functions $f_k$, the sets $M^\tau$ are nonempty.         We will need also the averaging operators $I^\tau$ over $M^\tau$:
        $$
        I^\tau f := \bbE(f|M^\tau) = \frac{1}{m(M^\tau)}\int_{M^\tau}f(t)\,dt.
        $$

        By assumption~\eqref{linf_proof_approx}, we have
        \begin{equation}
            \label{linf_proof_approx2}
            m\{\|\vec{f}(t)-\vec{h}(t)\|_{\ell_\infty^N} \le 1/2-\eps\}=1.
        \end{equation}
 Consider the vector $I^\tau\vec{f}:=(I^\tau f_1,\ldots,I^\tau f_N)$ for some $\tau$. Again, by independence,
        $$
        I^\tau f_k = \bbE(f_k|M^\tau) = \bbE(f_k|M_k^{\tau_k}).
        $$
        Hence, $I^\tau f_k \le 0$ if $\tau_k=0$, and $I^\tau f_k \ge 1-\delta$ if $\tau_k=1$.
        Consequently,
        $$
        \conv\{I^\tau\vec{f}\colon \tau\in\{0,1\}^N\}\supset[0,1-\delta]^N.$$
        By~\eqref{linf_proof_approx2}, all the vectors $I^\tau\vec{f}$, $\tau\in\{0,1\}^N$, are well approximated by the vectors $I^\tau\vec{h}$, which, as is
        easy to see, lie in the $n$-dimensional subspace $\{Av\colon v\in\mathbb R^n\}$.
        Hence,
        $$
        d_n([0,1-\delta]^N,\ell_\infty^N)
        \le d_n(\{I^\tau\vec{f}\colon \tau\in\{0,1\}^N\},\ell_\infty^N)
        \le 1/2 - \eps.
        $$
        On the other hand, by~\eqref{width_K_K} and~\eqref{width_cube}, we have
        $$
        d_n([0,1]^N,\ell_\infty^N) \ge
        (1/2)d_n([-1,1]^N,\ell_\infty^N) = 1/2.$$ As a result, since
        $d_n([0,1-\delta]^N,\ell_\infty^N)= (1-\delta)d_n([0,1]^N,\ell_\infty^N)$, we obtain:
        $(1-\delta)/2 \le 1/2-\eps$, which yields a contradiction for $\eps>\delta/2$. Since $\delta>0$ is arbitrary,
        the assertion is proved.
    \end{proof}

    Now we can state the following refinement of Theorem \ref{extth}.

     \begin{cor}
        Let $f_1,\ldots,f_N\in L_p[0,1]$ be normalized mean zero independent functions. Then the following
        estimates hold:
    \label{cor_lq}
        \begin{itemize}
            \item for $p=1$: $d_{N(1-\eps)}^\avg(\{f_1,\ldots,f_N\},L_1)_1 \ge c(\eps)$ for any $\eps>0$;
            \item for $p\in(1,2]$: $d_n^\avg(\{f_1,\ldots,f_N\},L_p)_2 \ge c(p)\cdot(1-n/N)^{1/2}$;
            \item for $p=\infty$: $d_{N-1}(\{f_1,\ldots,f_N\},L_\infty) \ge 1/2$.
        \end{itemize}
    \end{cor}
    \begin{proof}
        The case $p=1$ was treated in~\cite[Corollary 3.1]{M24}.
    Since the space $L_p$ for $1<p\le 2$ satisfies all the conditions of Theorem~\ref{main1} (see
    Sections~\ref{Prel3} and \ref{Prel4}), we obtain the second assertion.
        Finally, the case $p=\infty$ is the content of Proposition~\ref{prop_linf}.

    \end{proof}

    Note that for $p\in(1,2]$, a similar but somewhat weaker estimate 
        $$
        d_n^\avg(\{f_1,\ldots,f_N\}, L_p)_p\ge c(p)(1-n/N)^{1/p}
    $$
  was obtained in \cite{MR25}   (see \S1.2 therein).

  Next, we proceed with the Orlicz spaces (see Section~\ref{Prel_symm}).
    By $L\log^\alpha L$, $\alpha>0$, we denote the space generated by an Orlicz function equivalent, for large $u$,
    to the function $u\ln^\alpha u$.

    \begin{cor}
    \label{cor1a}
    Suppose that
    \begin{eqnarray}
    \label{eq1}
    \Phi(au)\le K_1a^2\Phi(u)
    \end{eqnarray}
    for some $K_1>0$ and all $a\ge 1$, $u\ge 1$. Moreover, let the conjugate function $\Phi'$ satisfy the
    condition:
    \begin{eqnarray}
    \label{eq2}
    \Phi'(u+v)\le K_2\Phi'(u)\Phi'(v)
    \end{eqnarray}
    for some $K_2>0$ and all $u\ge 1$, $v\ge 1$.

    Then there exists $c>0$, depending only on $\Phi$, such that for an arbitrary collection of mean zero independent functions
    $\{f_k\}_{k=1}^N$ from the Orlicz space $L_\Phi$, with $\|f_k\|_{L_\Phi}\ge 1$,
    $k=1,\dots,N$, and any $n<N$, the following inequality holds:
    $$
        d_n^\avg(\{f_1,\dots,f_N\},L_\Phi)_2\ge c(1-n/N)^{1/2}.
    $$

    In particular, this is true for the space $L\log^\alpha L$ if $\alpha\ge 1$.
    \end{cor}
    \begin{proof}
    First of all, inequality \eqref{eq1} guarantees that the space $L_\Phi$ satisfies a lower $2$-estimate
    \cite[Proposition on p.~118 and Theorem on p.~121]{KMP97}, and also that $L_\Phi\supset L_2$. Furthermore,
    $(L_\Phi)'=L_{\Phi'}$ (see, e.g., \cite[\S\,14]{KR}), and hence, by \eqref{eq2}, the space $(L_\Phi)'$ has the
    Kruglov property (see \cite{Kr70} or \cite[Theorem 10]{AS}). Thus, the first assertion of the corollary follows
    from Theorem~\ref{main1}.

    To prove the second assertion, it suffices to observe that the function $\Phi(u)$, equivalent to
    $u\ln^\alpha u$ for large $u$, satisfies condition \eqref{eq1}, while the
        conjugate function $\Phi'(u)$ in this case is
    equivalent to the function $e^{u^{1/\alpha}}$ \cite[Theorem 6.1]{KR}, and, as is easy to check, when
    $\alpha\ge 1$, it satisfies inequality \eqref{eq2}.

    \end{proof}

\begin{remark}
    \label{rem_kruglov}
    The condition $X'\in \bbK$ (respectively, the assumption that the function $\Phi'$ satisfies inequality
    \eqref{eq2}), in general, is not necessary in Theorem~\ref{main1} (respectively, in Corollary~\ref{cor1a}).
    We provide a corresponding example.

    As already mentioned in the proof of Corollary~\ref{cor1a}, if $\Phi(u)$ is
    an Orlicz function equivalent to the function $u\ln^\alpha u$ for large $u$, then
    $\Phi'(u)$ is equivalent to the function $e^{u^{p}}$,
    where $p=1/\alpha$. Clearly, $\Phi'$ does not satisfy condition \eqref{eq2}
    if $p>1$. Moreover, the
    exponential Orlicz space $\mathrm{Exp}L_p$, generated by the function $\Phi'$, does not possess the Kruglov property
    (see~Section~\ref{Prel4}). Nevertheless, it is known \cite[Theorem II.9]{Brav} that for $1<p\le 2$ the inequality
    \eqref{Bahr-Ess} still holds for $X'=\mathrm{Exp}L_p$. Therefore, by Theorem~\ref{main1}, the assertion of
    Corollary~\ref{cor1a} extends to the values $\alpha\ge 1/2$.
\end{remark}

    Finally, let us consider the Lorentz spaces.
    Since $L_{p,q}\in\bbK$ for all $1<p<\infty$ and $1\le q\le\infty$, $(L_{p,q})'=L_{p',q'}$ (where $p'$ and $q'$
    are the conjugate exponents to $p$ and $q$, respectively), $L_{p,q}\subset L_{p_1,q_1}$ for any
    $1<p_1<p<\infty$, $1\le q,q_1\le\infty$, and $L_{p,q}$ satisfies a lower $\max\{p,q\}$-estimate and an upper
    $\min\{p,q\}$-estimate (see, e.g., \cite{Creek} or \cite{D-01}), applying Theorem~\ref{main1}, we obtain

    \begin{cor}
    \label{Lor1}
    For any $1<p<2$, $1\le q\le 2$ there exists a constant $c_{p,q}>0$, depending only on $p$ and $q$, such that for
    any collection of mean zero independent functions $\{f_k\}_{k=1}^N$ from the space $L_{p,q}$, with  $\|f_k\|_{L_{p,q}}\ge 1$, $k=1,\dots,N$, and any $n<N$, we have
    $$
        d_n^\avg(\{f_1,\dots,f_N\},L_{p,q})_2\ge c_{p,q}(1-n/N)^{1/2}.
    $$
    \end{cor}

    \vskip0.1cm

\subsection{Spaces without the (IR) property}

    As noted in Section~\ref{Prel_symm}, $L_1$ and $L_\infty$ are the largest and the smallest s.s., respectively,
    while $L_2$ lies ``in the middle'' of the $L_p$-scale.
    Theorem~\ref{th_intro_phi} (see the Introduction) shows that all spaces lying between $L_2$ and $L_\infty$ that are
    in a certain sense separated from $L_2$ and are not equal to $L_\infty$ fail to have the property (IR). The construction, used in the
    following proof, strengthens the one from \cite[Proposition 4.2]{M24}.

    \begin{proof}[Proof of Theorem~\ref{th_intro_phi}.]
        We intend to construct families of mean zero independent functions in the given space $X$ that are well approximated by low-dimensional subspaces.
        More precisely, for a fixed $\gamma\in(0,1)$ and for sufficiently large positive integer $N$, we will find mean zero independent functions
        $f_1,\ldots,f_N\in X$, $\|f_k\|_{X}\ge 1$, as well as approximating functions $g_1,\ldots,g_N$ from an $n$-dimensional subspace of $X$, where
        $n:=\lfloor N\gamma\rfloor+1$, such that $\|f_k-g_k\|_{X}\le\gamma$ for all $k$.
        This will yield the following required width estimate:
        \begin{equation}
            \label{orlicz_width_a}
        d_{\gamma N+1}(\{f_1,\ldots,f_N\},X) \le \gamma.
        \end{equation}

     To obtain \eqref{orlicz_width_a}, we will use the width estimate~\eqref{width_b2} for the Euclidean ball $B_2^N$ in $\ell_\infty^N$.
  Denoting by $V_n^*$ the corresponding extremal $n$-dimensional subspace, we get        $$
        \rho(B_2^N,V_n^*)_\infty = d_n(B_2^N,\ell_\infty^N) \le C(\gamma) N^{-1/2}.
        $$
Let $M>0$ and $s^*\in\mathbb N$ be the parameters whose values we will choose later. Now, we describe the construction of the families $\{f_k\}$ and $\{g_k\}$.
        As $f_1$ we take a function, defined on $[0,1]$, with the distribution
        $$
        m\{f_1=0\}=1-\eps,\;m\{f_1=M\}=m\{f_1=-M\}=\eps/2,
        $$
        where $\eps>0$ is determined by the normalization condition: $\|f_1\|_{X}=1$. Since $f_1$ is equimeasurable with the function $M\chi_{[0,\eps]}$, we obtain:
        \begin{equation}
            \label{orlicz_eps_a}
        M\cdot \phi_X(\eps)=1.
        \end{equation}
        We take $f_2,\ldots,f_N$ to be independent copies of $f_1$ and consider the vector-function
        $\vec{f}=(f_1,\ldots,f_N)$.
         
        Denote by $\mathcal A_s$ the set of those $t\in[0,1]$ for which the vector $\vec{f}(t)$ has exactly $s$
        non-zero coordinates.
        The approximating vector $\vec{g}=(g_1,\ldots,g_N)$ is defined, for $t\in\mathcal A_s$, $s\le s^*$, as the best $\ell_\infty^N$-approximation of the vector $\vec{f}$ in the subspace $V_n^*$:
        $$
        \vec{g}(t):=\arg\min_{y\in V_n^*}\|\vec{f}(t)-y\|_{\ell_\infty^N},
        $$
        and $\vec{g}(t):=0$ for $t\in\mathcal A_s$, $s>s^*$.
        Since the vector $\vec{g}$ takes values only in the $n$-dimensional subspace $V_n^*\subset\mathbb{R}^N$, the
        functions $g_1,\ldots,g_N$ lie in an $n$-dimensional subspace of $X$.
        Let us estimate the norms $\|f_k-g_k\|_{X}$, $k=1,2,\dots,N$, from above.
        
        Setting
        $$
        \mathcal B:=\bigcup_{s=0}^{s^*} \mathcal A_s
        \quad\text{and}\quad
        \mathcal C:=\bigcup_{s=s^*+1}^N \mathcal A_s,
        $$
        we obtain
        \begin{equation}
            \label{orlicz_xi_eta_a}
            \|f_k-g_k\|_{X}\le \|(f_k-g_k)\chi_{\mathcal B}\|_{X}+\|(f_k-g_k)\chi_{\mathcal C}\|_{X}.
        \end{equation}
        We estimate the terms on the right-hand side of this inequality separately.

        Let $t\in\mathcal A_s$, $s\le s^*$. Then
        $$
        \|\vec{f}(t)\|_{\ell_2^N} \le Ms^{1/2},
        \quad \|\vec{f}(t)-\vec{g}(t)\|_{\ell_\infty^N} \le Ms^{1/2} C(\gamma) N^{-1/2},
        $$
       and hence, since $\|\chi_{[0,1]}\|_X=1$ (see Section~\ref{Prel_symm}), we have
       $$
       \|(f_k-g_k)\chi_{\mathcal B}\|_{X}\le M\sqrt{s^{*}} C(\gamma) N^{-1/2}.$$
      Consequently, for every $M>0$ and positive integer $N$, setting
       \begin{equation}
       \label{param_s}
           s^{*}:=\left\lfloor \frac{N\gamma^2}{4C(\gamma)^2M^2} \right\rfloor,
\end{equation}        
    we obtain
      \begin{equation}
       \label{first est}
          \|(f_k-g_k)\chi_{\mathcal B}\|_{X}\le \frac{\gamma}2.
\end{equation}

        Next, for $s>s^*$, from  relations~\eqref{orlicz_eps_a} and~\eqref{param_s} it follows that
        $$
        m(\mathcal A_s)
        \le \binom{N}{s}\eps^s \le (eN/s)^s \eps^s \le (eN\eps/s^*)^s
        \le \left(\frac{C_1(\gamma)N\eps M^{2}}{N}\right)^s
            =\left(\frac{C_1(\gamma)\eps}{\phi_X(\eps)^2}\right)^s.
        $$
       Moreover, by the assumption $\varlimsup_{t\to0} \phi_X(t)t^{-1/2}=\infty$,
     one can find $\eps>0$ such that
        \begin{equation*}
       \label{second est}
            \frac{C_1(\gamma)\eps}{\phi_X(\eps)^2}\le\frac12.
\end{equation*}
       From this stage we fix $\eps$ satisfying the last inequality and also fix $M$, for which equality~\eqref{orlicz_eps_a} holds. Then
        we obtain that $m(\mathcal A_s)\le 2^{-s}$ for $s>s^*$, whence $m(\mathcal C)\le 2^{-s^*}$.
        Thus, by the trivial estimate $\|\vec{f}(t)-\vec{g}(t)\|_{l_\infty^N} \le M$, $t\in\mathcal C$, we have
        $$
        \|(f_k-g_k)\chi_{\mathcal C}\|_{X}\le M\phi_X(2^{-s^*}).
        $$
       
       It is easy to verify (see also \cite[Lemma 2.3]{Abook}) that the condition $X\ne L_\infty$ implies that
       $\lim_{t\to 0}\phi_X(t)=0$. By~\eqref{param_s}, the parameter $s^*$ can be made arbitrarily large by taking $N$
       large enough, and hence one can guarantee that the inequality
       $$
       M\phi_X(2^{-s^*})\le \gamma/2
       $$
       holds. Then, by the previous inequality,
       $$
        \|(f_k-g_k)\chi_{\mathcal C}\|_{X}\le \frac{\gamma}2.
        $$
        Conbining this inequality together with estimates \eqref{first est} and \eqref{orlicz_xi_eta_a}, we get \eqref{orlicz_width_a}. 
        Finally, since $\gamma$ is arbitrary, $X\not\in(\mathrm{IR})$ (see the Introduction).
    \end{proof}
    
   \begin{cor}
   \label{ass space}
   There exists an s.s. $X$ such that $X\not\in\IR$ and $X'\not\in\IR$.
   \end{cor}
   
  \begin{proof}
  It is easy to verify that the function $\psi_0$ defined by the relations:
  $\psi_0(t)=t^{1/2}\ln(e^2/t)$, if $0<t\le 1$, and $\psi_0(0)=0$, is increasing and concave on $[0,1]$.
  A sequence of numbers $\{t_n\}_{n=0}^\infty$ satisfying the conditions
  $1=t_0>t_1>t_2>\dots>0$ and $\lim_{n\to\infty}t_n=0$ will be chosen later. We define a continuous function $\psi$ on
  $[0,1]$ as follows: $\psi(t_n)= \psi_0(t_n)$, $n=0,1,\dots$, and if $t_{n+1}<t\le t_n$, $n=0,1,\dots$, then $\psi(t)$
  is linear.
  If $t_n\to 0$ sufficiently fast, then there exist points $\tau_n\in (t_{n+1},t_n)$, $n=0,1,\dots$, such that
\begin{equation}
\label{ass1}
\lim_{n\to\infty}\psi(\tau_n)\tau_n^{-1/2}=0.
\end{equation}

Clearly, $\psi(t)$ is increasing, $\psi(t)/t$ is decreasing on $[0,1]$, and $\psi(0)=0$.
Consequently (see \cite[Theorem~II.4.7]{KPS}), $\psi$ is the fundamental function of some s.s. $X$. Moreover, since by
definition
$$
\lim_{n\to\infty}\psi(t_n)t_n^{-1/2}=\lim_{n\to\infty}\ln(e^2/t_n)=\infty,
$$
applying Theorem~\ref{th_intro_phi}, we obtain $X\not\in\IR$.

On the other hand, if $X'$ is the space associated to $X$, then $\phi_{X'}(t)=t/\psi(t)$ (see \cite[equality
(4.39)]{KPS}). Hence, from \eqref{ass1} it follows that
$$
\lim_{n\to\infty}\phi_{X'}(\tau_n)\tau_n^{-1/2}=
\lim_{n\to\infty}\frac{\tau_n}{\psi(\tau_n)\tau_n^{1/2}}=
\lim_{n\to\infty}\frac{\tau_n^{1/2}}{\psi(\tau_n)}=\infty.
$$
Thus, $X'\not\in\IR$ again by Theorem~\ref{th_intro_phi}.
   \end{proof}
   
   \begin{probl}
   \label{pr1}
 Suppose $X$ is such that $X\in\IR$ and $X'\in\IR$. Is it true that then $X=L_2$ (up to  equivalence of norms)?
    \end{probl}

    Let us consider a more general situation where $X$ is a quasi-Banach symmetric space, in which, in contrast to a (Banach) s.s. (see Section~\ref{Prel_symm}), the metric is defined by a quasi-norm that satisfies the quasi-triangle inequality with some constant $C\ge 1$. For the definition, properties, and examples of such spaces, see, e.g., \cite{ASZ-14}.

    The following statement uses definitions and notation from Section~\ref{Prel3}.

    \begin{theor}
    \label{th_LFR}
    Suppose that a quasi-Banach s.s. $X$ on $[0,1]$ has the Kruglov property and
    $\LFR(X)\cap (2,\infty]\ne\emptyset$.

    Then there exist $\delta=\delta(X)>0$ and a constant $C=C(X)>0$ such that for each $N\in\mathbb{N}$ one can find
    a family of symmetrically distributed independent functions $\{f_k\}_{k=1}^N$, $\|f_k\|_X=1$, $k=1,\dots,N$, 
    such that $n\ge N^{1-\delta}$ and
    $$
    d_n(\{f_1,\dots,f_N\},X)\le C N^{-\delta}.
    $$

    In particular, the assertion holds if $0<\alpha_X<1/2$.
    \end{theor}

    \begin{proof}
    By the hypothesis, there exists $q>2$ such that for each $N\in\mathbb{N}$, $X$ contains pairwise disjoint non-negative
    functions $x_k\in X$, $k=1,2,\dots,N$, for which inequality~\eqref{first} holds.
        In particular, the norms of $x_k$ in $X$ are bounded away from zero: $\|x_k\| \ge K^{-1}$,  $k=1,2,\dots,N$.
     
    Next, let the functions $f_k$ be symmetrically distributed, independent, and such that $|f_k|$ has the same
    distribution as the function $x_k$ for each $k=1,\dots,N$. Then, since the space $X$ has the Kruglov property, by
    inequality \eqref{main ineq1} (see also Theorem~21 in \cite{ASZ-14}),
    $$
    \Big\|\sum_{k=1}^N a_kf_k\Big\|_X\le C(X)\Big\|\sum_{k=1}^N a_kx_k\Big\|_X.$$
    From this and \eqref{first} it follows that
        \begin{equation}
            \label{fk_lq}
            \Big\|\sum_{k=1}^N a_kf_k\Big\|_X\le KC(X)\|(a_k)\|_{\ell_q^N}
        \end{equation}
        for all $N\in\mathbb{N}$ and $a_k\in\mathbb{R}$.

        Consider the linear operator $T\colon \ell_q^N\to X$ defined by the equalities $Te_k = f_k$, $k=1,\ldots,N$.
        Inequality~\eqref{fk_lq} means that $\|T\|\le KC(X)$. Applying relation~\eqref{width_T} and the width estimate~\eqref{width_b1a} for the octahedron $B_1^N$, we get
        \begin{multline*}
        d_n(\{f_1,\ldots,f_N\}, X)
        \le \|T\|\cdot d_n(\{e_1,\ldots,e_N\}, \ell_q^N) = \\
        = \|T\|\cdot d_n(B_1^N, \ell_q^N)
            \le KC(X)C(q)N^{-\delta},
            \quad \text{for}\;\;n\ge N^{1-\delta},\;\delta=\delta(q).
        \end{multline*}
        Finally, we can pass from $\{f_k\}$ to the normalized functions $\{f_k/\|f_k\|\}$; since
        $\|f_k\|=\|x_k\|\ge K^{-1}$, this increases the width by at most a factor depending only on $K$.
        This proves the first assertion of the theorem.
        
   Since the condition $\alpha_X>0$ guarantees that the space $X$ has the Kruglov property
   (see~Section~\ref{Prel4}), and $1/\alpha_X\in\LFR(X)\cap (2,\infty)$ (see~Section~\ref{Prel3}), the proof of the theorem is complete.
    \end{proof}

    In particular, for quasi-Banach Lorentz spaces we obtain the following result.

    \begin{cor}
    \label{Lor3}
        Let $0<p<\infty$, $0<q\le\infty$. If $\max\{p,q\}>2$, then there exist $\delta=\delta(p,q)>0$ and a constant
        $C=C(p,q)$ such that for each $N\in\mathbb{N}$ one can find a family of symmetrically distributed independent
        functions $\{f_k\}_{k=1}^N$, $\|f_k\|_{p,q}=1$, $k=1,\dots,N$, such that
        for $n\ge N^{1-\delta}$ we have
    $$
    d_n(\{f_1,\dots,f_N\},L_{p,q})\le CN^{-\delta}.
    $$
    \end{cor}
    \begin{proof}
    Since the space $L_{p,q}$ has the Kruglov property for all $0<p<\infty$ and $0< q\le\infty$ \cite{ASZ-14}, the
        assertion follows from Theorem~\ref{th_LFR} and the fact that $p,q\in \LFR(L_{p,q})$
        (see \cite[Proposition 1 and Theorem 6]{D-01}).
    \end{proof}


    \begin{remark}
        It is easy to show that $\LFR(L_{p,q})\subset [\min\{p,q\},\max\{p,q\}]$. Moreover, if $1\le p\le q$, then
        $\LFR(L_{p,q})=\{p,q\}$ (see~\cite{A-23}).
    \end{remark}

   \begin{probl}
    \label{pr2}
       Does the space $L_{2,q}$ for $q<2$ has the (IR) property? Although this
     space satisfies a lower $2$-estimate, it does not contain $L_2$. Equivalently, since $q'=q/(q-1)>2$, we have:
     $(L_{2,q})'=L_{2,q'}\not\subset L_2$. Consequently,
    inequality \eqref{Bahr-Ess} fails for $X=L_{2,q}$ if $q<2$ \cite[Corollary 40]{ASZ-14}, and Theorem~\ref{main1} is not
    applicable in this case. On the other hand, one can easily see that the hypothesis of the last corollary also does not hold.
    \end{probl}

\begin{cor}
Let $\Phi$ be an Orlicz function such that for some $\eps,K_1,K_2>0$ the following conditions hold:
    \begin{equation}
        \label{orlicz_eq1a}
        \Phi(au)\ge K_1 a^{2+\eps}\Phi(u),\quad \forall\,a,u\ge 1,
    \end{equation}
\begin{equation}
\label{orlicz_eq2a}
\Phi(u+v)\le K_2 \Phi(u)\Phi(v),\quad \forall\,u,v\ge 1.
\end{equation}
    Then there exist $\delta>0$ and $C>0$, depending on $\Phi$, such that for each $N\in\mathbb{N}$ one can find
    a family of symmetrically distributed independent functions $\{f_k\}_{k=1}^N$, $\|f_k\|_{L_\Phi}=1$,
    $k=1,\dots,N$, such that for all $n\ge N^{1-\delta}$ it holds
    $$
    d_n(\{f_1,\dots,f_N\},L_\Phi)\le CN^{-\delta}.$$
    \end{cor}

    \begin{proof}
     First, an Orlicz space that satisfies condition~\eqref{orlicz_eq2a} has the Kruglov property.
        According to formula~\eqref{index}, condition~\eqref{orlicz_eq1a} is equivalent to the fact that the lower Boyd
        index $\alpha_{L_\Phi}$ is less than $1/2$. Therefore, since $1/\alpha_{L_\Phi}\in\LFR(L_\Phi)$ (see~Section~\ref{Prel3}), the required result follows from Theorem~\ref{th_LFR}.
    \end{proof}

    Note that this corollary applies also to some Orlicz spaces with $\alpha_{L_\Phi}=0$, for example, if
    $\Phi(u)=e^u-1$.

    Thus, $L_\Phi\not\in(\mathrm{IR})$ if conditions~\eqref{orlicz_eq1a} and~\eqref{orlicz_eq2a} hold.
    From Theorem~\ref{th_intro_phi} we obtain this conclusion under a substantially weaker condition.

   \begin{cor}
        \label{cor_orlicz}
        Let $\Phi$ be an Orlicz function such that
        $$
        \varlimsup_{u\to\infty}\Phi(u)u^{-2} = \infty.
        $$
        Then $L_\Phi\not\in(\mathrm{IR})$.
    \end{cor}

    \begin{proof}
        Since $\phi_{L_\Phi}(t)=1/\Phi^{-1}(1/t)$, where $\Phi^{-1}$ is the inverse function of $\Phi$ (see, e.g.,
        \cite[Theorem II.9.5]{KR}), the assumption of the corollary can be rewritten as
$$
        \varlimsup_{u\to\infty}{\phi_{L_\Phi}(t)}t^{-1/2} = \infty.
        $$
        Hence, it suffices to apply Theorem~\ref{th_intro_phi}.
   \end{proof}

\begin{remark}
 \label{rem Orlicz}
     We give an example of Orlicz function that does not satisfy the conditions of neither
     Corollary~\ref{cor1a} nor Corollary~\ref{cor_orlicz}. 
     
Following \cite[pp.~235-238]{JMShT} (see also \cite{H_R-S}), set
$$
\psi(u):=\sum_{k=1}^\infty (1-\cos(2^{-k}u\pi)),\;\;u\in\mathbb{R},$$
and for $q\in\mathbb{R}$
 $$
 \Phi_q(t):=t^2\exp(q\psi(\ln t))\;\text{if}\;t>0\;\;\text{and}\;\;\Phi_q(0)=0.$$
Clearly, $ \Phi_q(t)\ge 0$, $ \Phi_q(1)=1$. It is also known \cite[Lemma 1.1]{H_R-S} that the condition
    $|q|<1/(3\pi)$ guarantees that the function $\Phi_q$ is increasing and convex.
    
    Thus, in this case $ \Phi_q$ is an Orlicz function on $[0,\infty)$. We also assume that $q<0$. Then
    $\Phi_q(t)\le t^2$ for all $t\ge 0$, and hence the condition of Corollary~\ref{cor_orlicz} fails. We prove a
    similar statement concerning condition \eqref{eq1} of Corollary~\ref{cor1a}.

Suppose the contrary, i.e., that for some $C>0$ and all $a\ge 1$, $u\ge 1$ the following inequality holds:
\begin{equation}
\label{aux est1}
\Phi_q(au)\le Ca^2\Phi_q(u).
\end{equation}

After simple transformations we arrive at the relation
$$
\exp(-q(\psi(\ln u)-\psi(\ln(au)))\le C,$$
or, equivalently,
\begin{equation}
\label{aux est}
\psi(s)\le C'+\psi(s+t),
\end{equation}
 where the constant $C'$ does not depend on $s,t>0$.  
 
 For a fixed $s>0$, set $t=2^{k_0}-s>0$, where $k_0\in\mathbb{N}$. Since $s+t=2^{k_0}$, we obtain
\begin{multline*}
    \psi(s+t)= \psi(2^{k_0})
    = \sum_{k=1}^{k_0-1} (1-\cos(2^{k_0-k}\pi))+\sum_{k=k_0}^\infty (1-\cos(2^{k_0-k}\pi)) = \\
    = 2\sum_{k=k_0}^\infty \sin^2(2^{k_0-k-1}\pi)
    \le 2\pi^2\sum_{k=k_0}^\infty 2^{2(k_0-k-1)} =\frac{2\pi^2}{3}.
\end{multline*}
Therefore, from \eqref{aux est} it follows that $\psi(s)\le C'+\frac{2\pi^2}{3}$ for all $s>0$, i.e., the function
    $\psi$ is bounded from above on $(0,\infty)$. However, as is easy to verify that $\psi(s_n)\ge n$, where $s_n=4^n+4^{n-1}+\dots +1$, $n=1,2,\dots$. Thus, inequality \eqref{aux est}, and     hence also \eqref{aux est1}, is false.
    \end{remark}    
    
  \begin{probl}
    Find necessary and sufficient conditions on an Orlicz function $\Phi$ that
      imply that the Orlicz space $L_\Phi$ has the property (IR).
  \end{probl}

\vskip0.5cm

\subsection{Rigidity under additional assumptions}

Recall (see~\S\ref{Prel_symm}) that for every s.s. $X$ the normalization condition $\|\chi_{[0,1]}\|_X=1$ is assumed. Hence, for $f\in L_\infty$, we have the inequality $\|f\|_X\le \|f\|_\infty$,  which will be used in the proof of the
next statement.

      \begin{proposition}
    \label{pr3}
    Let an s.s. $X$ on $[0,1]$ satisfy a lower $2$-estimate, $X'\in\bbK$, and let $\delta>0$ be given.
     Assume also that $f_1,\ldots,f_N$ are mean zero independent  functions from $X$ such that $\|f_k\|_X\ge 1$, $k=1,2,\dots,N$.

          Then for all $n<N$ the estimate
        $$
        d_n^\avg(\{f_1,\dots,f_N\},X)_2
        \ge B^{-1}(1-n/N)^{1/2}
        $$
          holds, where the constant $B$ depends only on $X$, $\delta$, and the quantity $R_{1-\delta}(f)$ defined by
          the relation
        $$
          R_{1-\delta}(f):=\sum_{k=1}^N m\{|f_k|>1-\delta\}.
        $$
       \end{proposition}

    \begin{proof}
   Set $E_k:=\{|f_k|>1-\delta\}$.
   Letting $h_k:=f_k\chi_{E_k}$, for each $k=1,2,\dots$ we find a function
   $g_k'\in X'$ such that $\int_0^1 h_kg_k'\,dt=\|h_k\|_X$, $\supp g_k'\subset E_k$, and $\|g_k'\|_{X'}=1$.

   As before (see the proof of Theorem~\ref{th_intro_rigid}), set
 $$
    g_k'':=\mathbb{E}_{\mathcal{A}_k}g_k'-\mathbb{E}g_k',\;\;k=1,2,\dots,$$
    where $\mathbb{E}_{\mathcal{A}_k}$ is the conditional expectation operator with respect to the $\sigma$-algebra of
    subsets of $[0,1]$ generated by the function $f_k$. Then the functions $g_k''$, $k=1,2,\dots$, are independent,
    and moreover,
$$
\mathbb{E}g_k''=0,\;\|g_k''\|_{X'}\le 2,\;k=1,2,\dots,\;\text{and}\;\int_0^1 f_jg_k''\,dt=0,\;\text{if}\;j\ne k.$$
    Furthermore,
    $$
    \int_0^1 f_kg_k''\,dt=\int_0^1 \mathbb{E}_{\mathcal{A}_k}(f_kg_k')\,dt=\int_0^1f_kg_k'\,dt
    =\int_0^1 h_kg_k'\,dt=\|h_k\|_X,\;\;k=1,2,\dots$$
    Since $\|f_k\|_X=1$, it is easy to see that $\|h_k\|_X\ge\delta$, $k=1,\ldots,N$.

    If we now set $g_k:=g_k''/\|h_k\|_X$, the preceding equalities show that $\{g_k\}_{k=1}^\infty$ is a biorthogonal
    system to the sequence $\{f_k\}_{k=1}^\infty$. In view of Theorem \ref{main1}, it suffices to prove the inequality
    \begin{equation}
    \label{upp est}
    \Big\|\sum_{k=1}^N a_kg_k\Big\|_{X'}\le B \Big(\sum_{k=1}^N a_k^2\Big)^{1/2}.
    \end{equation}

 We represent $g_k=u_k+v_k$, where
 $u_k:=\mathbb{E}_{\mathcal{A}_k}g_k'/\|h_k\|_X$, $v_k:= -\mathbb{E}g_k'/\|h_k\|_X$, $k=1,2,\dots$.
 Note that $\supp u_k\subset E_k$,
$$
 \|u_k\|_{X'}=\|\mathbb{E}_{\mathcal{A}_k}g_k'\|_{X'}/\|h_k\|_X\le\|g_k'\|_{X'}/\|h_k\|_X
 \le \delta^{-1},
$$
and
 $$
 \|v_k\|_{L_\infty}\le |\mathbb{E}g_k'|/\|h_k\|_X\le\|g_k'\|_{X'}/\|h_k\|_X
\le \delta^{-1}.
$$
Since the space $X'$ has the Kruglov property, $g_k$ are mean zero and independent functions from $X'$, we have
    \begin{equation}
    \label{upp est 2}
    \Big\|\sum_{k=1}^N a_kg_k\Big\|_{X'}\le C(X) \Big\|\sum_{k=1}^N a_k\bar{g}_k\Big\|_{Z_{X'}^2},
    \end{equation}
    where $\bar{g}_k(t)=g_k(t-k+1)\chi_{[k-1,k]}(t),$ $k=1,2,\dots,N$, and the constant $C(X)$ depends only on $X$
    (see inequality \eqref{main ineq1} in Section~\ref{Prel4}).

        Assume that $R:= R_{1-\delta}(f)>1$ (if $R\le 1$, the argument simplifies).
As is well known \cite[\S~II.4.3, Corollary 1]{KPS}, the dilation operator
$\tilde{\sigma}_\tau x (t):=x(t/\tau)$ is bounded in any s.s. $Y$ on the semi-axis and
$\|\tilde{\sigma}_\tau\|_{Y\to Y}\le\max\{1,\tau\}$, $\tau>0$. Furthermore, since
$m(\supp\tilde{\sigma}_{R^{-1}}\bar{u}_k)\le R^{-1} m(E_k)$, we have
$$
  \sum_{k=1}^N m(\supp\tilde{\sigma}_{R^{-1}}\bar{u}_k)
  \le R^{-1}\sum_{k=1}^N m(E_k)\le 1.
$$
  Therefore, by the definition of the quasi-norm~\eqref{main ineq1dop}, taking into
  account that $\|v_k\|_{L_\infty}\le \delta^{-1}$, $k=1,\dots,N$, we obtain for any $a_k\in \mathbb{R}$
  \begin{multline}
    \label{finish}
    \Big\|\sum_{k=1}^N a_k\bar{g}_k\Big\|_{Z_{X'}^2} =
    \Big\|\tilde{\sigma}_{R}\Big(\sum_{k=1}^N a_k\tilde{\sigma}_{R^{-1}}\bar{g}_k\Big)\Big\|_{Z_{X'}^2}
    \le R \Big\|\sum_{k=1}^N a_k\tilde{\sigma}_{R^{-1}}\bar{g}_k\Big\|_{Z_{X'}^2} \le \\
      \le R \Big(\Big\|\sum_{k=1}^N a_k\tilde{\sigma}_{R^{-1}}\bar{u}_k\Big\|_{Z_{X'}^2}
      + \Big\|\sum_{k=1}^N a_k\tilde{\sigma}_{R^{-1}}\bar{v}_k\Big\|_{Z_{X'}^2} \Big)\le  \\
    \le R\Big(\Big\|\Big(\sum_{k=1}^N a_k\tilde{\sigma}_{R^{-1}}\bar{u}_k\Big)^*\Big\|_{X'}
        + \delta^{-1}\max_{k=1,\dots,N}|a_k|\cdot\|\chi_{[0,1]}\|_{X'}
        + \Big\|\sum_{k=1}^N a_k\bar{v}_k\Big\|_{L_2}\Big).
   \end{multline}
    Since by hypothesis $X$ satisfies a lower $2$-estimate, the space $X'$ admits an upper $2$-estimate with the same constant
    \cite[Proposition~1.f.5]{LT}. Observe also that the functions $\tilde{\sigma}_{R^{-1}}\bar{u}_k$,
    $k=1,2,\dots,N$, are pairwise disjoint, and hence the sum $\sum_{k=1}^N a_k\tilde{\sigma}_{R^{-1}}\bar{u}_k$ is
    equimeasurable with the sum $\sum_{k=1}^N a_kw_k$, where the functions $w_k\in X'$ are pairwise disjoint and
    equimeasurable with the functions $\tilde{\sigma}_{R^{-1}}\bar{u}_k$ for each $k$. Thus, by virtue of the
    inequality
    $$
    \|w_k\|_{X'}=\|(\tilde{\sigma}_{R^{-1}}\bar{u}_k)^*\|_{X'}\le
    \|\bar{u}_k^*\|_{X'}=\|{u}_k\|_{X'}\le \delta^{-1},
    $$
     we obtain
\begin{equation*}
  \Big\|\Big(\sum_{k=1}^N a_k\tilde{\sigma}_{R^{-1}}\bar{u}_k\Big)^*\Big\|_{X'}
= \Big\|\sum_{k=1}^N a_kw_k\Big\|_{X'}\le C_{X} \Big(\sum_{k=1}^N
    a_k^{2}\|w_k\|_{X'}^{2}\Big)^{1/{2}}\le \delta^{-1}C_{X}\Big(\sum_{k=1}^N a_k^{2}\Big)^{1/{2}},
    \end{equation*}
    where $C_{X}$ is the constant from the lower $2$-estimate of the space $X$.
 Moreover,
 $$
\Big\|\sum_{k=1}^N a_k\bar{v}_k\Big\|_{L_2}\le \delta^{-1}\Big\|\sum_{k=1}^N
a_k\chi_{(k-1,k]}\Big\|_{L_2}= \delta^{-1}\Big(\sum_{k=1}^N a_k^{2}\Big)^{1/{2}}.$$
 Thus, from~\eqref{finish} it follows that
    \begin{equation*}
        \Big\|\sum_{k=1}^N a_k\bar{g}_k\Big\|_{Z_{X'}^2} \le \delta^{-1}R(C_{X}+2)\Big(\sum_{k=1}^N
        a_k^{2}\Big)^{1/{2}}.
    \end{equation*}
    Finally, it remains to observe that this inequality and \eqref{upp est 2} imply estimate \eqref{upp est} with the constant $B=\delta^{-1}C(X)R_{1-\delta}(f)(C_{X}+2)$.
    \end{proof}

     \begin{cor}
    \label{support}
 Let an s.s. $X$ on $[0,1]$ satisfy a lower $2$-estimate and $X'\in\bbK$. Suppose that  $f_k$ are mean zero independent functions such that $\|f_k\|_X\ge 1$, $k=1,2,\dots$, and
$$
R_0(f):=\sum_{k=1}^\infty m(\supp f_k)<\infty.$$

        Then for any $N\in\mathbb{N}$ and $n<N$ we have
        $$
        d_n^\avg(\{f_1,\dots,f_N\},X)_2
        \ge B^{-1}(1-n/N)^{1/2},
        $$
        where the constant $B$ depends only on $X$ and the quantity $R_0(f)$.
    \end{cor}

      \begin{cor}
    \label{Lor4}
  If $1\le q<2$, the assertion of Proposition~\ref{pr3} holds for the space $X=L_{2,q}$.
    \end{cor}

    \begin{proposition}
        Let $X$ be a quasi-Banach s.s.
        For any $\eta>0$ there exists $\gamma=\gamma(X,\eta)>0$ such that for any family of independent functions
        $\{f_k\}_{k=1}^N\subset X$ satisfying the conditions
        \begin{equation}
            \label{I_eta}
            m\{f_k\ge\eta\}\ge\eta
            \quad\text{and}\quad
            m\{f_k\le-\eta\}\ge\eta,
            \quad k=1,\ldots,N,
        \end{equation}
        the inequality
    $$
    d_{\gamma N}(\{f_1,\dots,f_N\},X)\ge \gamma
    $$
    holds.
    \end{proposition}

    In particular, families of identically distributed functions possess the rigidity property in every quasi-Banach
    s.s.

    \begin{proof}
        Consider the following metric on the space of measurable functions:
        $$
        \rho(f,g):=\inf\{\eta\ge 0\colon m\{|f-g|>\eta\}\le\eta\}.
        $$
        Hypothesis~\eqref{I_eta} guarantees that
        \begin{equation}
            \label{rho_condition}
            \inf_{c\in\mathbb R}\rho(f_k,c\cdot\chi_{[0,1]})\ge\eta,\quad k=1,\ldots,N.
        \end{equation}
        At the same time, by~\cite[Corollary 3.3]{M24}, there exists $\beta=\beta(\eta)>0$ such that for any family
        of independent functions $\{f_k\}_{k=1}^N$ satisfying condition~\eqref{rho_condition} the inequality
        $$
        d_{\beta N}(\{f_1,\ldots,f_N\},\rho) \ge \beta
        $$
        holds.

        Since $X$ is symmetric, there exists $\delta=\delta(X,\beta)>0$ such that the inequality
        $\rho(f,g)\ge\beta$ implies $\|f-g\|_X \ge \delta$. From this we obtain the required lower bound for the
        width in $X$ with $\gamma:=\min\{\beta,\delta\}$.
    \end{proof}

\end{document}